\documentclass{amsart}

\input xy
\xyoption{all}

\usepackage{geometry}
\usepackage{amsmath}
\usepackage{amssymb}
\usepackage{cite}
\usepackage{amsthm}  
\usepackage{tikz-cd}
\usetikzlibrary{er,positioning}

\newtheorem{same}{This should never appear}[section]
\newtheorem{defin}[same]{Definition}

\newtheorem{remark}[same]{Remark}
\newtheorem{theorem}[same]{Theorem}
\newtheorem{example}[same]{Example}
\newtheorem{lemma}[same]{Lemma}
\newtheorem{fact}[same]{Fact}

\newtheorem{cor}[same]{Corollary}
\newtheorem{prop}[same]{Proposition}

\newtheorem{nota}[same]{Notation}

\newtheorem{probl}[same]{Problem}

\newtheorem{defin*}{Definition}
\newtheorem*{theorem*}{Theorem}
\newtheorem*{theorem1}{Theorem \ref{i-eq}}
\newtheorem*{theorem2}{Theorem \ref{main-f}}

\makeatletter
\newcommand{\skipitems}[1]{%
  \addtocounter{\@enumctr}{#1}%
}

\newcommand{\rest}{\mathord{\upharpoonright}}
\newcommand{\id}{\textrm{id}}

\newcommand{\K}{\mathbf{K}}

\newcommand{\Rm}{R\text{-Mod}}

\newcommand{\leap}[1]{\le_{#1}}

\newcommand{\lea}{\leap{\K}}

\title{On limit models and parametrized noetherian rings}
\date{AMS 2020 Subject Classification:
Primary:  13L05, 03C48. Secondary: 03C45, 03C60.} 

\author{Marcos Mazari-Armida}
\email{marcos\_mazari@baylor.edu}
\urladdr{https://sites.baylor.edu/marcos\_mazari/}
\address{Department of Mathematics \\ Baylor University \\ Waco, Texas, USA}
\thanks{The author's research was partially supported by an NSF grant DMS-2348881, an AMS-Simons Travel Grant 2022--2024, and a Simons Foundation grant: Travel Support for Mathematicians.}

\begin{document}

\begin{abstract}

We study limit models in the abstract elementary class of modules with embeddings as algebraic objects. We characterize parametrized noetherian rings using the degree of injectivity of certain limit models.

We show that the number of limit models and how close a ring is from being noetherian are inversely proportional.

\begin{theorem} Let $n \geq 0$ The following are equivalent.

\begin{enumerate}
\item $R$ is left $(<\aleph_{n } )$-noetherian but not left $(< \aleph_{n -1 })$-noetherian.
\item The abstract elementary class of modules with embeddings has exactly $n +1$ non-isomorphic $\lambda$-limit models for every $\lambda \geq (\operatorname{card}(R) + \aleph_0)^+$ such that the class is stable in $\lambda$. 
\end{enumerate}
\end{theorem}

We further show that there are rings such that the abstract elementary class of modules with embeddings has exactly $\kappa$ non-isomorphic $\lambda$-limit models for every infinite cardinal $\kappa$.


\end{abstract}


\maketitle



 \section{Introduction}
 
 Limit models were introduced by Kolman and Shelah \cite{kosh} in the context of non-elementary model theory. They were introduced as a substitute for saturation in abstract elementary classes. Intuitively a \emph{limit model} is a universal model with some level of homogeneity (see Definition \ref{limit}). Limit models have proven to be  a key concept in the study of abstract elementary classes as witnessed by  for example \cite{shvi}, \cite{shelahaecbook}, \cite{grvavi}, \cite{grva}, \cite{vasey18} and \cite{leu2}.

 
 An abstract elementary class (AEC for short) is determined by a class of objects, in this paper \emph{all} modules, and a class of  embeddings, in this paper \emph{all} embeddings. Since we are only working in the abstract elementary class of modules with embeddings, there is no need to introduce what an abstract elementary class is and instead we work purely algebraically.   Furthermore, we introduce all the AEC notions used in this paper in the context of the  AEC of modules with embeddings in Section 2.2.


The first objective of this paper is to study limit models, in the AEC of modules with embeddings, as algebraic objects.  Limit models have been studied algebraically in many classes of modules  \cite{maz20}, \cite{kuma}, \cite{m2}, \cite{m3}, \cite{maz2}, \cite{m-tor} and \cite{mj}. Nevertheless, previous results only characterized the \emph{shortest} limit model and the \emph{long} limit models.  In this paper, we go further and we are able to understand \emph{all}  limit models.

The two algebraic notions that come into play when studying limit models are those of  $\kappa$-injective module \cite{eksa} and of $(<\kappa)$-noetherian ring (for $\kappa$ an infinite cardinal).  The first is a generalization of injective module while the second is a generalization of noetherian ring.

The key algebraic results are: the longer a limit model is the closer it is from being injective (Lemma \ref{k-inj}) and if a limit model is not injective it is because there is an ideal with many generators (Lemma \ref{non-inj}). Putting together these two results we obtain a characterization of  parametrized noetherian rings using limit models. 

\begin{theorem1} Assume $\kappa$ is a regular cardinal. The following are equivalent.
\begin{enumerate}
\item $R$ is left $(<\kappa)$-noetherian.
\item The $(\lambda, \kappa)$-limit model is $\kappa^+$-injective for every $\lambda$ such that the AEC of modules with embeddings is stable in $\lambda$ and $ \kappa \leq \lambda$. 
\item The $(\lambda, \kappa)$-limit model is $\kappa^+$-injective for some $\lambda$ such that the AEC of modules with embeddings is stable in $\lambda$ and $ \kappa \leq \lambda$.
\end{enumerate}
\end{theorem1}

 The second objective of this paper is to show that there are rings such that the AEC of modules with embeddings has many non-isomorphic limit models. The existence of such rings provide the \emph{first examples} of  abstract elementary classes with exactly $\kappa$ non-isomorphic $\lambda$-limit models for $\kappa > 2$.

 We show that the number of limit models in the AEC of modules with embeddings and how close a ring is from being noetherian are inversely proportional. The following result provides further evidence that many model theoretic notions have algebraic substance in the context of ring theory. The equivalence for $n=0$ was first obtained in \cite[3.12]{m2}.

 \begin{theorem2}Let $n \geq 0$ The following are equivalent.

\begin{enumerate}
\item $R$ is left $(<\aleph_{n } )$-noetherian but not left $(< \aleph_{n -1 })$-noetherian.\footnote{If $n=0$, this should be understood as  $R$ is  left $(<\aleph_{0 } )$-noetherian, i.e., $R$ is left noetherian. }
\item The abstract elementary class of modules with embeddings has exactly $n +1$ non-isomorphic $\lambda$-limit models for every $\lambda \geq (\operatorname{card}(R) + \aleph_0)^+$ such that the class is stable in $\lambda$. 
\end{enumerate}
 
\end{theorem2}

 We further show that the forward direction holds for every infinite cardinal $\kappa$ (Lemma \ref{inf}). More precisely, if $R$ is left $(<\aleph_{\kappa +1 } )$-noetherian but not left $(< \aleph_{\kappa })$-noetherian, then the AEC of modules with embeddings has exactly $\kappa$ non-isomorphic $\lambda$-limit models for every $\lambda \geq (\operatorname{card}(R) + \aleph_0)^+$ such that the class is stable in $\lambda$. 
 
 The proofs of Theorem \ref{main-f} and Lemma \ref{inf} rely on the key algebraic results mentioned before Theorem \ref{i-eq} and on Bumby's classical result \cite{bumby}  that injective modules which are embeddable in one another are isomorphic.

 Besides providing examples of AECs with many non-isomorphic limit models. The current paper strengthens the need to continue the study of strictly stable AECs started in \cite[\S 4, \S 5]{sh394}, \cite{tamenessone}, \cite{bovan}, \cite{vaseyt}, \cite{leu2}. Furthermore, the paper provides important intuition on what to expect in the strictly stable setting. The author hopes to use this intuition in future work.


 The paper is divided into two sections. Section 2 has the basic results and presents the notions of abstract elementary classes used in this paper. Section 3 has the main results. The first half of the section deals with the algebraic treatment of limit models and the second half uses these results to show that there are many non-isomorphic limit models if the ring is far from being noetherian.
 
 I would like to thank  Daniel Herden, Betty Montero, Jonathan Feigert and Manfred Dugas for discussions that helped to develop the paper. In particular, I would like to thank Daniel Herden for providing Example \ref{m-ex}. I would also like to thank Jeremy Beard, Jan Trlifaj, Sebastien Vasey, Wentao Yang, and an anonymous referee for helpful comments that helped improve the paper.

\section{Basic Results and Notions}
\subsection{Module theory} All rings considered in this paper are associative with unity, all modules are left $R$-modules and all ideals are left ideals unless specified otherwise. We denote infinite cardinals by $\kappa, \lambda, \mu$ and by $\kappa^+$ the least cardinal greater than $\kappa$. For a ring $R$, we will write $\text{card}(R)$ for its cardinality. 

\begin{defin} Let $\kappa$ be an infinite cardinal. An ideal $I$ is \emph{$(<\kappa)$-generated} if there is $\mu < \kappa$ and  $\{ a_i : i <\mu \} \subseteq I$ such that $I = \langle \{ a_i : i <\mu \} \rangle$.  $I$ is \emph{strictly $\kappa$-generated} if it is $(< \kappa^+)$-generated but not $(< \kappa)$-generated. 
\end{defin}

A cardinal $\mu$ is regular if the cofinality of $\mu$ is $\mu$.  Recall that $\aleph_0$ is regular and $\kappa^+$ is regular for every infinite cardinal $\kappa$.

\begin{prop}\label{small-strict} Let $R$ be a ring, $I$ an ideal, and $\kappa$ an infinite cardinal.
If $I$ is a strictly $\kappa$-generated ideal, then for every regular cardinal $\mu < \kappa$ there is $J$ a strictly $\mu$-generated ideal such that $J \subseteq I$. 
\end{prop}
\begin{proof}
Let $\{a_\alpha : \alpha <\kappa \} \subseteq I$ such that:

\begin{enumerate}
\item $I = \langle \{ a_\alpha  : \alpha <\kappa \} \rangle$.
\item  for every $\alpha < \kappa$, $a_{\alpha } \notin I_\alpha := \langle \{ a_\beta : \beta < \alpha  \} \rangle$. 
\end{enumerate} 

This can be achieved since $I$ is a strictly $\kappa$-generated ideal. 

Let $J = I_\mu = \langle \{ a_\alpha : \alpha < \mu \} \rangle$. Assume for the sake of contradiction that $I_\mu$ is $(< \mu)$-generated. Then there is $\theta < \mu$ and $\{ b_j : j < \theta \} \subseteq I_\mu$ such that $I_\mu = \langle \{ b_j : j < \theta \} \rangle$. Since $I_\mu = \bigcup_{\alpha  < \mu} I_\alpha$ and $\mu$ is regular, there is $\alpha < \mu$ such that $ \{ b_j : j < \theta \} \subseteq I_\alpha$. This is a contradiction as $a_{\alpha  } \in  \langle \{ b_j : j < \theta \} \rangle$ but $a_{\alpha } \notin I_{\alpha }$ by Condition (2) of the construction. 
\end{proof}

The following cardinal  originally appears in \cite{ekl}.
\begin{defin}
Given a ring $R$, let $\gamma(R)$ be the minimum infinite cardinal such that every left ideal of $R$ is $(< \gamma(R))-$generated.  
\end{defin}

\begin{remark}
It is clear that  $\gamma(R) \leq \operatorname{card}(R)^+$. 
\end{remark}

We introduce parametrized Noetherian rings, which have been considered in, for example \cite[3.2.5]{ragr}, \cite[6.30, 8.12]{gotr} and \cite[\S 3.2]{cox}.

\begin{defin}
A ring  $R$ is left $(<\kappa)$-noetherian if every left ideal is  $(<\kappa)$-generated, i.e., $\gamma(R)  \leq \kappa$.
\end{defin}

 A left $(<\aleph_0)$-noetherian ring is precisely a left noetherian ring. 
 
 We provide a couple of examples.

\begin{example}\label{m-ex} Let $\alpha$ be an ordinal. Let $R= \mathbb{Q}[ x_i : i < \aleph_{\alpha} ]$, the polynomial ring over $\mathbb{Q}$ with commuting variables $\{ x_i : i < \aleph_{\alpha} \}$. $R$ is $(<\aleph_{\alpha +1})$-noetherian as $| R| = \aleph_\alpha$ and $R$ is not $(< \aleph_\alpha)$-noetherian as $\langle x_i : i < \aleph_{\alpha}   \rangle$ is not $(< \aleph_\alpha)$-generated.

\end{example}

\begin{example}[{ \cite[3.3.(1)]{ls}}]\label{ls}
Let $\Gamma$ be the set of all well-ordered subsets of the interval $[0, \infty)$ with respect to the usual order of the reals. Let $F$ be a field and let $x$ be a commuting variable. Let $R$ denote the set of all \emph{formal power series} of the form

\[
\sum_{\delta \in \Delta} a_{\delta} x^{\delta}
\]

where $a_{\delta} \in F$ and $\Delta \in \Gamma$. Then we can make $R$ into a ring with respect to addition and multiplication defined in the obvious way.

The only ideals of $R$ are of the form $\langle x ^{\alpha} \rangle$, which is finitely generated, and

\[
\langle x^{>\alpha} \rangle :=\left\{u x^{\beta} : \beta>\alpha \text { and } u \text { is unit }\right\} \cup\{0\},
\]

which is strictly countably generated as it is generated by $\{ x^{ \alpha + \frac{1}{n}} : n \geq 1 \} \cup \{0\}$. 

Therefore, $R$ is a left $(<\aleph_1)$-noetherian ring which is not left $(<\aleph_0)$-noetherian.

\end{example}

\begin{remark}
Volatile rings were introduced in  \cite[\S 3]{ls} as a class of rings opposite to noetherian rings. Example \ref{ls} shows that volatile rings are not far  from being noetherian from our perspective as Example \ref{ls} provides a volatile ring which is left $(<\aleph_1)$-noetherian. 
\end{remark}

We expect the existence of rings that are strictly $(<\aleph_\alpha)$-noetherian for limit ordinals $\alpha$, but we do not know any examples.

\begin{probl}\label{p2}
Show that for every (some) limit ordinal $\alpha$ there is a ring which is left $(<\aleph_\alpha)$-noetherian, but not left $(<\aleph_\beta)$-noetherian for every $\beta <\alpha$.
\end{probl}

We introduce a parametrized version of  injective module following \cite{eksa}.

 \begin{defin} Let $\mu$ be an infinite cardinal. $M$ is $\mu$-injective if for every ideal $I$  which is $(<\mu)$-generated and every homomorphism $f: I \to M$, there is a homomorphism $g: R \to M$ such that the following diagram commutes:
 
 \[\begin{tikzcd}
I \arrow[r, hook] \arrow[d, "f"'] & R \arrow[ld, "g", dashed] \\
M                                 &                          
\end{tikzcd}\]
 \end{defin}

 The following result follows from Baer's Criterion for injectivity \cite{baer}.
 
\begin{fact} For every $\mu \geq \gamma(R)$. 
$M$ is injective if and only if $M$ is $\mu$-injective. In particular for every $\mu \geq \operatorname{card}(R)^+$,  $M$ is $\mu$-injective if and only if $M$ is injective. 
\end{fact} 

\begin{remark}\
\begin{itemize}
\item It is clear that if $\mu_1 < \mu_2$ and  $M$ is $\mu_2$-injective, then $M$ is $\mu_1$-injective. 
\item  $\aleph_1$-injective modules are sometimes referred to as $\aleph_0$-injective modules in the literature, see for example \cite{mom}.

\end{itemize}

\end{remark}

\begin{remark}
A model theoretic characterization of $\mu$-injective modules is presented in \cite[\S 4]{eksa}. We will not work with it in this paper, but we introduce it for the convenience of the model theoretic audience.  $M$ is $\mu$-injective if every $\mu$-divisible system in $M$ consistent with $M$ has a solution in $M$. A $\mu$-divisible system in $M$ consists of a set of formulas of the following form $\{ r_i x = a_i : i < \beta \}$ such that  $\beta < \mu$, and $a_i \in M$ and $r_i \in R$ for every $i < \beta$. A $\mu$-divisible system $\Delta$ is consistent with $M$ if there is $N$ and $n \in N$ such that $M \subseteq_R N$ and $n \in N$ is a solution of $\Delta$.

\end{remark}

\begin{fact}[{ \cite[3.10]{eksa}}]\label{al_0}
Every direct sum of $\aleph_0$-injective modules is $\aleph_0$-injective. 
\end{fact}

We finish this subsection by using parametrized injective modules to characterize noetherian rings. We will later use them to characterize parametrized noetherian rings. 

\begin{lemma}\label{many-equiv} The following are equivalent.
\begin{enumerate}
\item $R$ is left noetherian.
\item Every $\aleph_0$-injective modules is injective. 
\item Every $\aleph_0$-injective module is $\aleph_1$-injective. 
\item Direct sums of $\aleph_1$-injective modules are $\aleph_1$-injective.

\item For some $\kappa \geq \aleph_1$, direct sums of $\kappa$-injective modules are $\kappa$-injective. 
\item For some $\kappa \geq \aleph_1$, direct sums of countably many $\kappa$-injective modules are $\kappa$-injective.
\item For some $\kappa \geq \aleph_1$, direct sums of countably many $\kappa$-injective modules are $\aleph_1$-injective.
\item For every $\kappa \geq \aleph_1$, direct sums of $\kappa$-injective modules are $\kappa$-injective. 
\end{enumerate}
\end{lemma}
\begin{proof}\
$(1) \Rightarrow (2)$: It follows from the fact that every ideal is finitely generated. 

$(2) \Rightarrow (3)$: Clear.

$(3) \Rightarrow (4)$: Let $\{ M_i : i \in I \}$ such that $M_i$ is $\aleph_1$-injective for every $i \in I$. Then $M_i$ is $\aleph_0$-injective for every $i \in I$. As $\aleph_0$-injective modules are closed under direct sums by Fact \ref{al_0}, $\bigoplus_{i \in I} M_i$ is $\aleph_0$-injective. Therefore, $\bigoplus_{i \in I} M_i$ is $\aleph_1$-injective by $(3)$.

$(4) \Rightarrow (5) \Rightarrow (6) \Rightarrow (7)$: Clear.

$(7) \Rightarrow (1)$: Let $\kappa \geq \aleph_1$ such that (7) holds.

 Assume that $R$ is not left noetherian, then there is an ideal $I$ which is not finitely generated. So we can build $\{J _n : n < \omega \}$ an strictly increasing chain of finitely generated ideals contained in $I$ and let $J = \bigcup_{n < \omega} J_n$. Observe that $J$ is a strictly countably generated ideal and that $J/J_n \neq 0$ for every $n < \omega$. 

For each $n < \omega$, let $M_n$ be a  $\kappa$-injective module and $s_n: J/J_n \to M_n$ be an embedding.  These exists because every module can be embedded into an injective module (see also Corollary \ref{ec}). Let $s: \bigoplus_{n < \omega} J/J_n \to \bigoplus_{n < \omega } M_n$ be given by $(a_n)_{n < \omega } \mapsto (s_n(a_n))_{n < \omega}$

Let $h: J \to \bigoplus_{n < \omega} J/J_n$ be given by $a \mapsto (\pi_n(a))_{n < \omega}$ where $\pi_n$ is the projection of $ J$ onto $J/ J_n$. Let $f = s \circ h : J \to \bigoplus_{n < \omega } M_n$. Then by (7), it follows that $\bigoplus_{n < \omega } M_n$ is $\aleph_1$-injective. As $J$ is countably generated, there is $g: R \to \bigoplus_{n < \omega } M_n$ such that $f \subseteq g$. 

Let $g(1)=(b_n)_{n < \omega }$. One can show that $b_n \neq 0$ for every $n < \omega$. This contradicts the fact that $g(1) \in  \bigoplus_{n < \omega } M_n$. Therefore, $R$ is left noetherian.

$(8) \Rightarrow (6)$: Clear. 

$(2) \Rightarrow (8)$: Similar to  $(3) \Rightarrow (4)$. 
\end{proof}

\begin{remark}\
\begin{itemize}
\item The equivalence between (1), (2) and (3) already appears in \cite[3.18, 3.24]{eksa}. 
\item The  equivalence between (1) and  (5) is a parametrized version of the classical result that a ring is left noetherian if and only if direct sums of injective modules are injective, see for example \cite[3.46]{lam2}
\end{itemize}

\end{remark}

\subsection{Abstract elementary classes of modules} Abstract elementary classes (AECs for short) were introduced by Shelah \cite{sh88} to study classes of structures axiomatized in infinitary logics. An AEC is a pair $\K=(K, \lea)$ where $K$ is a class of objects in a fixed language and $\lea$ is a partial order on $K$ extending the substructure relation such that $\K$ satisfies certain axioms. See for example \cite[4.1]{baldwinbook09} for the definition.

In this paper, we will only work with the abstract elementary class of left $R$-modules with embeddings and we will denote it by $(R\text{-Mod}, \subseteq_R)$.\footnote{The language of this AEC is  $\{0, +,-\} \cup \{ r\cdot   : r \in R \}$ where  $r \cdot$ is interpreted as multiplication by $r$ for each $r \in R$.}  Due to this, we will work purely algebraically and we will introduce AEC notions in the specific framework of the AEC of modules with embeddings. A more general introduction to AECs from an algebraic point of view is given in \cite{m4}.

Given a module $M$, $|M|$ is the underlying set of $M$ and $\| M\|$ is the cardinality of $M$. In the rest of the section, $M, N, \cdots$ are left $R$-modules and $\lambda \geq \operatorname{card}(R) + \aleph_0$.

\begin{nota}\
\begin{itemize}
\item  If $\lambda$ is a cardinal, then $R\text{-Mod}_{\lambda}=\{ M  : M \text{ is a module and }\| M \|=\lambda \}$ and $R\text{-Mod}_{\leq \lambda}=\{ M  : M \text{ is a module and }\| M \| \leq\lambda \}$
\item We write $f: M \xrightarrow[A]{} N$ if $f$ is a homomorphism and $f\rest_A = \id_A$. 
\end{itemize}

\end{nota}

\begin{defin}
$M$ is \emph{universal over} $N$ if and only if $\| M\|= \| N\|$, $N \subseteq_R M$
 and for any $N^* \in \Rm_{\| M \|}$ such that
$N \subseteq_R N^*$, there is $f: N^* \xrightarrow[N]{} M$ an embedding.
\end{defin}

Recall that an increasing chain  $\{ M_i : i < \alpha\}\subseteq R\text{-Mod}$ (for $\alpha$ an ordinal) is a \emph{continuous chain} if $M_i =\bigcup_{j < i} M_j$  for every $i < \alpha$ limit ordinal. We introduce limit models \cite{kosh}.

\begin{defin}\label{limit}
Let $\lambda$ be an infinite cardinal and $\alpha < \lambda^+$ be a limit ordinal.  $M$ is a \emph{$(\lambda,
\alpha)$-limit model over} $N$ if and only if there is $\{ M_i : i <
\alpha\}\subseteq \Rm_\lambda$ an increasing continuous chain such
that:
\begin{itemize}
\item $M_0 =N$ and  $M= \bigcup_{i < \alpha} M_i$, and
\item $M_{i+1}$ is universal over $M_i$ for each $i <
\alpha$.
\end{itemize}

$M$ is a $(\lambda, \alpha)$-limit model if there is $N \in
\Rm_\lambda$ such that $M$ is a $(\lambda, \alpha)$-limit model over
$N$. $M$ is a $\lambda$-limit model if there is a limit ordinal
$\alpha < \lambda^+$ such that $M$  is a $(\lambda,
\alpha)$-limit model.  $M$ is a limit model if there is an infinite cardinal $\lambda$ such that $M$ is a $\lambda$-limit model.

\end{defin}

\begin{fact}\label{same-cof} Let $\lambda$ be an infinite cardinal and $\alpha, \beta < \lambda^+$  limit ordinals.
If $M$ is a $(\lambda, \alpha)$-limit model, $N$ is a $(\lambda, \beta)$-limit model and $cf(\alpha)=cf(\beta)$, then $M$ and $N$ are isomorphic.
\end{fact}

Due to the previous fact for every $\alpha < \lambda^+$ limit ordinal there is a unique $(\lambda, \alpha)$-limit model. 

Using limit models we introduce a non-standard definition of stability.\footnote{The following definition  is equivalent to the notion of stable in $\lambda$ in the sense of bounding the number of Galois types for AECs with amalgamation, joint embedding and no maximal models  \cite[2.9]{tamenessone}. $(R\text{-Mod}, \subseteq_R)$ has the three properties just mentioned, so it is equivalent in our setting.} 

\begin{defin} $(R\text{-Mod}, \subseteq_R)$ is stable in $\lambda$ if there is a $\lambda$-limit model.\end{defin}

Moreover, if $(R\text{-Mod}, \subseteq_R)$ is stable in $\lambda$, then for every $M  \in \Rm_\lambda$ there is $N \in \Rm_\lambda$ such that $N$ is universal over $M$ (see for example \cite[2.9]{tamenessone}). In particular, if $(R\text{-Mod}, \subseteq_R)$ is stable in $\lambda$, for every $\alpha < \lambda^+$ limit ordinal there is a $(\lambda, \alpha)$-limit model. 

\begin{fact}[{\cite[3.8]{m2}}]\label{st-no} Assume $\lambda \geq (\operatorname{card}(R) + \aleph_0)^+$.
  $(R\text{-Mod}, \subseteq_R)$ is stable in $\lambda$ if and only if $\lambda^{<\gamma(R)}=\lambda$. 
\end{fact}

It follows from the previous result that for any ring $R$ there are always $2^{\operatorname{card}(R) + \aleph_0}$-limit models.

Finally, we say that \emph{$M$ is universal in $\Rm_\lambda$}  if $\| M\| = \lambda$ and for every $N \in \Rm_\lambda$ there is an embedding $f:N\rightarrow M$. 

\begin{fact}[{\cite[2.10]{maz20}}]\label{uni-l} Let $\lambda$ be an infinite cardinal. If $M$ is a $\lambda$-limit model, then $M$ is universal in $\Rm_\lambda$.
\end{fact}

We will often use that if $f: M \to N$ is an embedding then there is $M'$ and $g: M' \cong N$ such that $M \subseteq_R M'$ and $f \rest M = g \rest M$. 

\section{Main results}

\subsection{Algebraic results} We begin by showing a technical result that we will use a couple of times in the paper.
\begin{prop}\label{disjoint-ap} If $M_1, M_2, N \in \Rm_{\leq \lambda}$, $f: M_1 \to N$ is a homomorphism and $M_1 \subseteq_R M_2$, then there are $L \in \Rm_\lambda$ and $g: M_2 \to N$ a homomorphism such that the following diagram commutes:

  \[
 \xymatrix{\ar @{} [dr]  M_2 \ar[r]^{g}  & L \\
 M_1 \ar[u]^{\subseteq_R} \ar[r]^{f} &  N \ar[u]_{\subseteq_R } 
 }
\]

  and $g[M_2] \cap N = f[M_1] (= g[M_1])$.
  
  Moreover, if $f$ is an embedding, then $g$ is an embedding.
\end{prop}
\begin{proof} We first take the pushout of $(f: M_1 \to N, i: M_1 \hookrightarrow M_2)$, i.e., let $P=(N \oplus M_2)/ \{ (f(m), -m ) : m \in M_1\}$, $g_1: N \to P$ given by $n_1 \mapsto [(n_1, 0)]$, and $g_2: M_2 \to P$ given $m_2 \mapsto [(0, m_2)] )$. One can show that $g_1 \circ f = g_2\rest_{M_1}$, $g_1$ is an embedding and  $g_1[N] \cap g_2[M_2] = g_1 \circ f[M_1]$.

Let $L$ and $h$ be such that the following diagram commutes:

\[
 \xymatrix{\ar @{} [dr]  L   \ar[r]_h^{\cong} & P \\
 N \ar[u]^{\subseteq_R} \ar[r]_{g_1}^{\cong} &   g_1[N]  \ar[u]_{\subseteq_R}
 }
\]

Let $g = h^{-1} \circ g_2 : M_2 \to L$. One can show that $g$ is as required. 

For the moreover part, if $f$ is an embedding, $g_2$ is an embedding and hence $g$ is an embedding. \end{proof}

We show a relation between limit models and parametrized injective modules. This result extends \cite[3.7]{m2}.

\begin{lemma}\label{k-inj} Assume $\delta < \lambda^+$ limit ordinal.
If $M$ is a $(\lambda, \delta)$-limit model, then $M$ is $\text{cf}(\delta)$-injective. In particular if $\text{cf}(\delta) \geq \gamma(R)$, then $M$ is injective. 
\end{lemma}
\begin{proof} Let $I$  be an ideal which is $(<\text{cf}(\delta))$-generated and $f: I \to M$ be a homomorphism. 

Let  $I = \langle \{ a_\alpha : \alpha < \mu \} \rangle$ for $\mu < \text{cf}(\delta)$. Fix $\{ M_i : i < \delta \}$ a witness to the fact that $M$ is a $(\lambda, \delta)$-limit model. As $\mu  < \text{cf}(\delta)$ there is $j <\delta$ such that $f[I] \subseteq_R M_j$. Then $I \subseteq_R R$ and $f: I \to M_j$ is a homomorphism. So there is $L \in \Rm_\lambda$ and $f_1$ such that the following diagram commutes:

  \[
 \xymatrix{\ar @{} [dr]  R \ar[r]^{f_1}  & L \\
 I \ar[u]^{\subseteq_R} \ar[r]^{f} &  M_{j} \ar[u]_{\subseteq_R } 
 }
\]

by Proposition \ref{disjoint-ap}.

Since $M_{j+1}$ is universal over $M_j$ there is $h: L \xrightarrow[M_j]{} M_{j+1}$. It is easy to show that $g = h \circ f_1: R \to M$ is such that $g \rest I = f \rest I$. Therefore, $M$ is $\text{cf}(\delta)$-injective.

\end{proof}

\begin{remark}
In particular, all limit models are $\aleph_0$-injective.
\end{remark}

\begin{lemma}\label{non-inj} Assume $(\Rm, \subseteq_R)$ is stable in $\lambda$ and $\kappa \leq \lambda$.
If $I$ is a strictly $\kappa$-generated ideal, then the $(\lambda, \kappa)$-limit model is not $\kappa^+$-injective. 
\end{lemma}
\begin{proof}
Let $\{a_\alpha : \alpha <\kappa \} \subseteq I$ such that:

\begin{enumerate}
\item $I = \langle \{ a_\alpha  : \alpha <\kappa \} \rangle$.
\item   for every $\alpha < \kappa$, $a_{\alpha } \notin I_\alpha := \langle \{ a_\beta : \beta < \alpha  \} \rangle$. 
\end{enumerate} 

This can be achieved since $I$ is a strictly $\kappa$-generated ideal. 

We build $\{ M_\alpha : \alpha < \kappa\} \subseteq \Rm_\lambda$ an increasing continuous chain and $\{ f_\alpha : \alpha < \kappa \}$ such that:

\begin{enumerate}
\item for every $\alpha < \kappa$,  $M_{\alpha +1}$ is universal over $M_\alpha$.
\item for every $\alpha < \kappa$, $f_{\alpha} : I_\alpha \to M_\alpha$ is an embedding.
\item if $\alpha < \beta < \kappa$, then $f_\alpha \subseteq f_\beta$.
\item for every $\alpha < \kappa$, the following diagram commutes
  \[
 \xymatrix{\ar @{} [dr]  I_{\alpha+1} \ar[r]^{f_{\alpha +1}}  & M_{\alpha + 1} \\
 I_\alpha \ar[u]^{\subseteq_R} \ar[r]^{f_{\alpha}} &  M_\alpha \ar[u]_{\subseteq_R} 
 }
\]

and $f_{\alpha +1}[I_{\alpha +1}] \cap M_{\alpha} = f_\alpha [I_\alpha]$.
\end{enumerate}

\fbox{Enough} Let $M = \bigcup_{\alpha < \kappa} M_\alpha$ and $f =  \bigcup_{\alpha < \kappa} f_\alpha$. $M$ is the $(\lambda, \kappa)$-limit model by Condition (1). Assume for the sake of contradiction that $M$ is $\kappa^+$-injective. Then there is $g$ such that the following diagram commutes:

\[\begin{tikzcd}
I \arrow[r, hook] \arrow[d, "f"'] & R \arrow[ld, "g", dashed] \\
M                                 &                          
\end{tikzcd}\]

Then $g(1) \in M =  \bigcup_{\alpha < \kappa} M_\alpha$, so there is $\alpha < \kappa$ such that $g(1) \in M_\alpha$. Now, $g(a_\alpha)  = f_{\alpha +1}(a_\alpha)$ as $a_\alpha \in I_{\alpha + 1}$ and $f \subseteq g$, so $g(a_\alpha) \in f_{\alpha + 1}[I_{\alpha+1}]$. Moreover, $g(a_\alpha)= a_\alpha g(1) \in M_\alpha$ as $g$ is a homomorphism and $g(1) \in M_\alpha$. So $g(a_\alpha) \in  f_{\alpha + 1}[I_{\alpha+1}] \cap M_\alpha$. Then by Condition (4) there is $b \in I_\alpha$ such that $g(a_\alpha)= f_\alpha(b)$. As $f \subseteq g$, $f(a_{\alpha}) = f(b)$. Since $f$ is an embedding by Condition (2), $a_\alpha = b$. This is a contradiction as $a_\alpha \notin I_\alpha$ and $b \in I_\alpha$. 

\fbox{Construction} In the base step take $M_0 \in \Rm_\lambda$ such that $I_0 \subseteq_R M_0$ and let $f_0 = id_{I_0}$ and in the limit steps take unions, so we are left with the successor steps.  Let $\alpha = \beta +1$. Applying Proposition \ref{disjoint-ap} to $f_\alpha: I_\alpha \to M_\alpha$ and $I_\alpha \subseteq_R I_{\alpha +1}$, there are  $L \in \Rm_\lambda$ and $g: I_{\alpha +1}  \to L$ such that the following diagram commutes:

  \[
 \xymatrix{\ar @{} [dr]  I_{\alpha +1} \ar[r]^{g}  & L \\
 I_\alpha \ar[u]^{\subseteq_R} \ar[r]^{f_\alpha} &  M_\alpha \ar[u]_{\subseteq_R } 
 }
\]

  and $g[I_{\alpha +1}] \cap M_\alpha = f_\alpha [I_\alpha]$.
  
  Since $(\Rm, \subseteq_R)$ is stable in $\lambda$, there is $M_{\alpha +1} \in \Rm_\lambda$ such that $M_{\alpha +1}$ is universal over $L$.   Let $f_{\alpha + 1}  = g: I_{\alpha +1} \to M_{\alpha + 1}$. It is clear that $f_{\alpha +1}$ is the required homomorphism and one can show that $M_{\alpha +1}$ is universal over $M_{\alpha}$ using Proposition \ref{disjoint-ap}.
\end{proof}

We show that limit models can be used to characterize parametrized noetherian rings. 

\begin{theorem}\label{i-eq}Assume $\kappa$ is a regular cardinal. The following are equivalent.
\begin{enumerate}
\item $R$ is left $(<\kappa)$-noetherian.
\item The $(\lambda, \kappa)$-limit model is $\kappa^+$-injective for every $\lambda$ such that $(\Rm, \subseteq_R)$ is stable in $\lambda$ and $ \kappa \leq \lambda$. 
\item The $(\lambda, \kappa)$-limit model is $\kappa^+$-injective for some $\lambda$ such that $(\Rm, \subseteq_R)$ is stable in $\lambda$ and $ \kappa \leq \lambda$.
\end{enumerate}
\end{theorem}
\begin{proof}

$(1) \Rightarrow (2)$:  Let $M$ be the $(\lambda, \kappa)$-limit model. $M$ is $\kappa$-injective by Lemma \ref{k-inj}. Hence $M$ is $\kappa^+$-injective as every ideal is $(<\kappa)$-generated.

$(2) \Rightarrow (3)$: Clear.

$(3) \Rightarrow (1)$: Assume for the sake of contradiction that $R$ is not left $(<\kappa)$-noetherian. Then there is $I$ a strictly $\kappa$-generated ideal by Proposition \ref{small-strict}. Let $\lambda \geq \kappa$ and $M$ be the $(\lambda, \kappa)$-limit model. $M$ is $\kappa^+$-injective by assumption. This is a contradiction as  $M$ is not $\kappa^+$-injective by Lemma \ref{non-inj}. 
\end{proof}

A similar argument to that of Theorem \ref{i-eq} can be used to characterize parametrized noetherian rings via parametrized injective modules.

\begin{lemma}\label{equiv}Assume $\kappa$ is a regular cardinal. The following are equivalent.
\begin{enumerate}
\item $R$ is left $(<\kappa)$-noetherian.
\item  Every $\kappa$-injective module is injective. 
\item Every $\kappa$-injective module is $\kappa^+$-injective. 
\end{enumerate}
\end{lemma}

\begin{remark}
The equivalence between (1) and (2) of Lemma \ref{equiv} generalizes \cite[3.18]{eksa}. Moreover, in \cite[Remark 2]{ekl} it is mentioned that the equivalence between (1) and (2) for an arbitrary regular cardinal $\kappa$ can be obtained using the methods of \cite[Theorem 2]{ekl}. \cite{ekl} does not provide an argument and the ideas used in that paper and this paper are very different. 
\end{remark}

The following follows directly from Theorem \ref{i-eq} and Fact \ref{uni-l}.

\begin{cor}\label{ec} Assume $\kappa$ is a regular cardinal.
If $R$ is not left $(<\kappa)$-noetherian, then every module $M$ can be embedded into a $\kappa$-injective module which is not $\kappa^+$-injective. 
\end{cor}

 The existence of a strictly $\aleph_\omega$-injective module is still not known.

\begin{probl}[{ \cite[Remark 2.(iii)]{ekl}}]
Show that there are rings with a $\aleph_\omega$-injective module which is not $\aleph_{\omega+1}$-injective. 
\end{probl}

A natural idea would be to construct $\kappa$-injective envelopes as an analogue of injective envelopes. The following result together with \cite{gks} shows that this is impossible for an envelope with reasonable properties.
\begin{lemma}
Assume $\kappa$ is regular. If $R$ is not left $(<\kappa)$-noetherian, then  are $M, N$ $\kappa$-injective modules and $f: M \to N$ and $g: N \to M$ embeddings such that $M$ is not isomorphic to $N$ 
\end{lemma}
\begin{proof}
Let $\lambda \geq \kappa + (\operatorname{card}(R) + \aleph_0)^+$ such that $\lambda^{<\gamma(R)} = \lambda$. Let $M$ be the $(\lambda, \kappa)$-limit model and $N$ be the $(\lambda, \gamma(R)^+)$-limit model. $M$ and $N$ can be embedded into one another by Fact \ref{uni-l}. They are not isomorphic as $M$ is not $\kappa^+$-injective by Lemma \ref{non-inj} but $N$ is $\kappa^+$-injective by Lemma \ref{k-inj}.
\end{proof}

\begin{remark}
The previous result also shows that the classical result for injective modules of Bumby \cite{bumby}  (Fact \ref{inj-iso}) fails for $\kappa$-injective modules.
\end{remark}

\subsection{The number of limit models} We focus on counting the number of limit models. The following result for injective modules will be useful.

\begin{fact}[{\cite{bumby}}]\label{inj-iso}  If $f: A \to B$, $g: B \to A$ are embeddings and $A, B $ are injective, then $A$ is isomorphic to $B$.
\end{fact}

We show that long limit models are isomorphic.
\begin{lemma}\label{long-lim}
  If $M$ is a $(\lambda, \alpha)$-limit model, $N$ is a $(\lambda, \beta)$-limit model and $cf(\alpha),cf(\beta) \geq \gamma(R)$, then $M$ and $N$ are isomorphic. 
\end{lemma}
\begin{proof}
$M$ and $N$ are injective by Lemma \ref{k-inj}. Since $M$ and $N$ are universal in $\Rm_\lambda$ by Fact \ref{uni-l}, there are $f: M \to N$ and $g: N \to M$ embeddings. Therefore, $M$ is isomorphic to $N$ by Fact \ref{inj-iso}.
\end{proof} 

We provide an upper bound to the number of limit models. In the following results, given an ordinal $\delta$ (possibly finite), $|\delta| + 1$ and $|\delta| +2$ should be interpreted as cardinal arithmetic addition.

\begin{lemma}\label{above} Let $\delta$ be an ordinal (possibly finite) such that $\aleph_\delta$ is regular and $\lambda$ be a cardinal such that $(\Rm, \subseteq_R)$ is stable in $\lambda$ and $\lambda \geq \aleph_{\delta}$. 
If $R$ is $(< \aleph_\delta)$-noetherian, then there at most $|\delta| + 1$ non-isomorphic $\lambda$-limit models.
\end{lemma}
\begin{proof}
For every $\alpha \leq \delta $, let $M_\alpha$ be the $(\lambda, \aleph_\alpha)$-limit model. We show  that every $\lambda$-limit model is isomorphic to $M_\alpha$ for some $\alpha \leq \delta $.

Assume $L$ is a $(\lambda, \beta)$-limit model. If $cf(\beta)= \aleph_\alpha$ for $\alpha < \delta$, then $L$ is isomorphic to $M_\alpha$ by Fact \ref{same-cof}. If $cf(\beta) \geq \aleph_\delta$ then $cf(\beta) \geq \gamma(R)$ as $R$ is $(<\aleph_{\delta})$-noetherian,  so $L$ is isomorphic to $M_{\delta}$ by Lemma \ref{long-lim} and the fact that $\operatorname{cf}(\aleph_\delta)= \aleph_\delta \geq  \gamma(R)$.
\end{proof}

\begin{remark}
If we do not assume that  $\aleph_\delta$ is regular and instead assume that $\lambda \geq \aleph_{\delta +1}$, the proof of the previous result can be modified to show that in that case there are also at most $|\delta| + 1$ non-isomorphic $\lambda$-limit models. This is achieved by  considering the $(\lambda, \aleph_{\delta +1})$-limit model instead of the $(\lambda, \aleph_{\delta})$-limit model.
\end{remark}

We provide a lower bound to the number of limit models.
 
\begin{lemma}\label{below} Let $\delta$ be an ordinal (possibly finite)  and $\lambda$ be a cardinal such that $(\Rm, \subseteq_R)$ is stable in $\lambda$ and $\lambda \geq \aleph_{\delta +1}$. 	 If $R$ is  not $(< \aleph_\delta)$-noetherian, then there at least $|\delta|+2$ non-isomorphic $\lambda$-limit models.
\end{lemma}
\begin{proof} Let $\Sigma= \{ \alpha \leq \delta + 1 : \aleph_{\alpha} \text{ is regular } \}$. Observe that $| \Sigma | = |\delta| + 2$.


For every $\alpha \in \Delta$, let $M_\alpha$ be the $(\lambda, \aleph_\alpha)$-limit model. We show that if
$\alpha \neq \beta \leq \delta + 1$ both in $\Sigma$, then $M_\alpha$ is not isomorphic to $M_\beta$.

 Assume without loss of generality that $\alpha < \beta \leq \delta + 1$. Since $R$ is  not $(< \aleph_\delta)$-noetherian then there is $I$ a strictly $\mu$-generated ideal for some $\mu \geq \aleph_\delta$. So there is $J$ a strictly $\aleph_\alpha$-generated ideal by Proposition \ref{small-strict} as $\aleph_\alpha$ is regular and $\aleph_\alpha \leq \aleph_\delta$ . Then $M_\alpha$ is not $\aleph_{\alpha+1}$-injective by Lemma \ref{non-inj}. Since $M_\beta$ is $\aleph_\beta$-injective by Lemma \ref{k-inj} and $\alpha < \beta$, we have that $M_\beta$ is $\aleph_{\alpha + 1}$-injective. Therefore, $M_\alpha$ is not isomorphic to $M_\beta$. \end{proof}

We show that the number of limit models and how close a ring is from being noetherian are inversely proportional. 

\begin{theorem}\label{main-f} Let $n \geq 0$ The following are equivalent.

\begin{enumerate}
\item $R$ is left $(<\aleph_{n } )$-noetherian but not left $(< \aleph_{n -1 })$-noetherian.\footnote{If $n=0$ this should be understood as  $R$ is left $(<\aleph_{0 } )$-noetherian, i.e., $R$ is left noetherian. }
\item $(R\text{-Mod}, \subseteq_R)$ has exactly $n +1 $ non-isomorphic $\lambda$-limit models for every $\lambda \geq (\operatorname{card}(R) + \aleph_0)^+$ such that $(R\text{-Mod}, \subseteq_R)$ is stable in $\lambda$. 
\item $(R\text{-Mod}, \subseteq_R)$ has exactly $n +1 $ non-isomorphic $2^{\operatorname{card}(R) + \aleph_0}$-limit models.
\end{enumerate}
\end{theorem}
\begin{proof}
The result follows directly from Lemma 3.14 and Lemma 3.16, the fact that ${(R\text{-Mod}, \subseteq_R)}$ is stable in $2^{\operatorname{card}(R) + \aleph_0}$ and the following observation:   if $(R\text{-Mod}, \subseteq_R)$ is stable in $\lambda \geq (\operatorname{card}(R) + \aleph_0)^+$ then $\lambda \geq \gamma(R)$ as $\lambda ^{<\gamma(R)} = \lambda$ by Fact \ref{st-no}.
\end{proof}

\begin{remark}
The equivalence for $n=0$ of the previous result was obtained in \cite[3.12]{m2}.
\end{remark}

For infinite cardinals we do not obtain an equivalence as in Theorem \ref{main-f}, but we  have the forward direction.

\begin{lemma}\label{inf} Let $\kappa$ be an infinite cardinal.
If $R$ is left $(<\aleph_{\kappa +1 } )$-noetherian but not left $(< \aleph_{\kappa })$-noetherian, then  $(R\text{-Mod}, \subseteq_R)$ has exactly $\kappa $ non-isomorphic $\lambda$-limit models for every $\lambda \geq (\operatorname{card}(R) + \aleph_0)^+$ such that $(R\text{-Mod}, \subseteq_R)$ is stable in $\lambda$. 

\end{lemma}

The backward direction of Lemma \ref{inf} fails since a ring which is left $(<\aleph_{\kappa +2 } )$-noetherian but not left $(< \aleph_{\kappa +1 })$-noetherian has exactly $\kappa $ non-isomorphic $\lambda$-limit models for every $\lambda \geq (\operatorname{card}(R) + \aleph_0)^+$ such that $(R\text{-Mod}, \subseteq_R)$ is stable in $\lambda$. An example of such a ring is given in Example \ref{m-ex}.

\begin{remark} Since $\kappa +1 = \kappa^+$ for every $\kappa$ a finite cardinal, when generalizing Theorem \ref{main-f} to $\kappa$ an infinite cardinal one might consider replacing \emph{$(< \aleph_{\kappa +1 })$-noetherian}  by  \emph{$(<\aleph_{\kappa^+ } )$-noetherian} in the statement of Lemma \ref{inf}. We will call this new statement the \emph{revised statement} in the next paragraph.

Both directions of the revised statement fail if  there are rings as those stated in Problem \ref{p2}, while one has an equivalence in the revised statement if those rings do not exist.  As we expect that there are rings as those stated in  Problem \ref{p2} we do not provide further details.\end{remark}

\begin{remark}
Theorem \ref{main-f} and Lemma \ref{inf} provide the first examples of abstract elementary classes with exactly $\kappa$ non-isomorphic $\lambda$-limit models for cardinal $\kappa > 2$. These examples strengthen the need to continue the study of strictly stable AECs. Furthermore, these examples provide important intuition on what to expect in the strictly stable setting. 
\end{remark}

We further classify the isomorphism types of the limit models in the AEC of modules with embeddings. We denote by $\operatorname{Reg}$ the class of regular cardinals and denote by $[\lambda_1, \lambda_2)$ the cardinals greater or equal than $\lambda_1$ and strictly less than $\lambda_2$.

\begin{lemma}\label{cla}
Let $\lambda$ be a cardinal such that $\lambda \geq \gamma(R)^+$ and $(R\text{-Mod}, \subseteq_R)$ is stable in $\lambda$.

If $\alpha, \beta  \in ([\aleph_0, \gamma(R)) \cap \operatorname{Reg}) \cup \{\gamma(R)^+\}$ and $\alpha \neq \beta$, then the $(\lambda, \alpha)$-limit model and the $(\lambda, \beta)$-limit model are not isomorphic. Furthermore, every $\lambda$-limit model is isomorphic to a $(\lambda, \alpha)$-limit model for some $\alpha \in  ([\aleph_0, \gamma(R)) \cap \operatorname{Reg}) \cup \{\gamma(R)^+\}$. 
\end{lemma}
\begin{proof}
A similar argument to those of Lemma \ref{below} and Lemma \ref{above} can be used to obtain the result. 
\end{proof}

\begin{remark} Under the assumptions of the previous result we have that if $\gamma(R)$ is regular then the $(\lambda, \gamma(R))$-limit model is isomorphic to the  $(\lambda, \gamma(R)^+)$-limit model. While if $\gamma(R)$ is singular then the $(\lambda, \gamma(R))$-limit model is isomorphic to the $(\lambda, cf(\gamma(R)))$-limit model for $cf(\gamma(R)) \in [\aleph_0, \gamma(R)) \cap \operatorname{Reg}$.

\end{remark}

Conjecture 1.1 of \cite{bovan} for $\lambda \geq \gamma(R)^+$ follows directly from Lemmas \ref{cla} and \ref{long-lim}.
 
 \begin{cor}
 Let $\lambda$ be a regular cardinal such that $\lambda \geq \gamma(R)^+$ and $(R\text{-Mod}, \subseteq_R)$ is stable in $\lambda$. Let

\[ \Gamma =  \{ \alpha < \lambda^+ : cf(\alpha) = \alpha \text{ and  the } (\lambda, \alpha)\text{-limit model is isomorphic to the } (\lambda, \lambda)\text{-limit model } \}. \]

Then  $\Gamma =[ \gamma(R), \lambda^+) \cap \operatorname{Reg}$. 
 \end{cor}

In  \cite[5.9]{mj} it was shown that if $R$ is left noetherian then the $(\lambda, \alpha)$-limit model is $M^{(\alpha)}$ where $M$ is the $(\lambda, (\operatorname{card}(R) + \aleph_0)^+)$-limit model. The next result shows that the previous result is sharp.

\begin{lemma}
If $R$ is not left noetherian and $M$ is the $(\lambda, (\operatorname{card}(R) + \aleph_0)^+)$-limit model for $\lambda \geq  (\operatorname{card}(R) + \aleph_0)^+$, then $M^{(\alpha)}$ is not the $(\lambda, \alpha)$-limit model for every $\alpha $ such that $cf(\alpha) > \omega$.
\end{lemma}
\begin{proof}
Basically repeat (7) implies (1) of Lemma \ref{many-equiv}; using that $M$ is universal in $\Rm_\lambda$ by Fact \ref{uni-l}. 
\end{proof}


\end{document}